\documentclass[11pt]{article}

\usepackage{amsmath,amsfonts,amssymb,amsthm,graphics,amscd,enumerate,verbatim,color}

\newtheorem{theorem}{Theorem}[section]
\newtheorem{definition}[theorem]{Definition}

\newtheorem{lemma}[theorem]{Lemma}
\newtheorem{corollary}[theorem]{Corollary}
\newtheorem{remark}[theorem]{Remark}

\def\B{{\cal{B}}}
\def\N{{\cal{N}}}
\def\R{{\cal{R}}}
\def\core{{\scriptsize{\textcircled{\#}}}}

\def\ind{{\rm\,ind\/}}

\begin{document}

\vspace{5cm}

\begin{center}
\Large{\bf Group, Moore-Penrose, core and dual core inverse \\ in rings with involution}\footnote{The authors are supported by the Ministry
of Science, Republic of Serbia, grant no. 174007.}
\end{center}

\begin{center}
{\bf Dragan S. Raki\'c \footnote{Corresponding author.}, Neboj\v{s}a \v{C}. Din\v{c}i\'c and Dragan S. Djordjevi\'c} 
\end{center}

\begin{abstract}

Let $R$ be a ring with involution. The recently introduced notions of the core and dual core inverse are extended from matrix to an arbitrary $*$-ring case. It is shown that the group, Moore-Penrose, core and dual core inverse are closely related and they can be treated in the same manner using appropriate idempotents. The several characterizations of these inverses are given. Some new properties are obtained and some known results are generalized. A number of characterizations of EP elements in $R$ are obtained. It is shown that core and dual core inverse belong to the class of inverses along an element and to the class of $(b,c)$-inverses. The case when $R$ is algebra of all bounded linear operators on Hilbert space is specifically considered.

\medskip

\smallskip

\end{abstract}

\section{Introduction}

Let $M_n$ be the algebra of all $n\times n$ complex matrices.
The Moore-Penrose inverse (MP inverse for short) of matrix $A$ is the unique matrix $A^\dag$ satisfying
$$(1)\, AA^\dag A=A \quad (2)\, A^\dag AA^\dag =A^\dag \quad (3)\, (AA^\dag)^*=AA^\dag \quad (4)\, (A^\dag A)^*=A^\dag A.$$
The inverse was introduced by Moore \cite{Moore} and latter rediscovered independently by Bjerhammar \cite{Bjerhammar} and Penrose \cite{Penrose}.
When $\ind(A)\leq 1$ i.e. $\text{rank}(A)=\text{rank}(A^2)$, the group inverse of $A$ (see \cite{BIG}) is unique matrix $A^\#$ defined by
$$(1)\, AA^\#A=A \quad (2)\, A^\#AA^\#=A^\# \quad (5) \, AA^\#=A^\#A.$$
Recently, Baksalary and Trenkler introduced in \cite{BT} a new pseudoinverse of a matrix named core inverse.
\begin{definition}\label{core matrix def}\,\, \cite{BT}\,\,
A matrix $A^\core\in M_n$ is the core inverse of $A\in M_n$ if it satisfies
\begin{equation}
AA^\core =P_A \text{ and } \R(A^\core)\subseteq \R(A).
\end{equation}
\end{definition}
Here $P_A$ stands for the orthogonal projection on $\R(A)$.
The core inverse exists if and only if $\ind(A)\leq 1$ in which case it is unique. In the same paper authors defined one more inverse, $\tilde{A}$, which is closely related to core inverse. We call this inverse dual core inverse of $A$ and denoted by $A_\core$. It is defined by, \cite{BT}
$$A_\core A=P_{A^*} \text{ and } \R(A_\core)\subseteq \R(A^*).$$

From now on $R$ denotes a ring with involution; we say $*$-ring for short. Our aim is to extend the definitions of these inverses to the case of $*$-ring. We will show that all four kinds of inverses can be treated in the similar way. The MP and group inverse of an element $a\in R$ are defined in the same way as in the matrix case; if they exist then they are unique. Some characterizations of the MP invertibility of an element of a ring is given in \cite{KolihaDjordjevicDragana}.

In Section 2 we will give an equivalent definition of the core inverse of matrix which serves us as a definition of core inverse of an element of a ring with involution: $x\in R$ is a core inverse of $a\in R$ if $$axa=a, \, xR=aR \, \text{ and } \, Rx=Ra^*.$$ The analogous alternative definitions for group, MP and dual core inverse of $a\in R$ are given (\ref{sve def na jednom mestu}).
In the theorems \ref{idempotenti za grup inverz}--\ref{idempotenti za dual core inverz}, we will characterize the existence of these inverses by the existence of idempotent $q\in R$ and self-adjoint idempotents $p,r\in R$ satisfying  $aR=qR$, $Ra=Rq$, $pR=aR$ and $Rr=Ra$. Namely, $a\in R$ is group invertible if and only if idempotent $q$ exists; $a$ is MP invertible if and only if $p$ and $r$ exist; $a$ is core invertible if and only if $p$ and $q$ exist; $a$ is dual core invertible if and only if $r$ and $q$ exist. Using these idempotents we obtain appropriate matrix representations for $a$, $a^\#$, $a^\dag$, $a^\core$ and $a_\core$.
We will characterize the core and dual core inverse by the set of equations in theorems \ref{idempotenti za core inverz} and \ref{idempotenti za dual core inverz}. This result is new even in the case $R=M_n$.
We will obtain a number of new properties and generalize most of the known properties of core inverse of complex matrix, that make sense in a $*$-ring.
We note that in the matrix case, the study of generalized inverses uses mainly finite dimensional linear algebra methods. In our setting of arbitrary $*$-ring, we can not use these methods.

In Section 3, the EP elements will be characterized.

In Section 4, we will show that considered inverses belong to the class of inverses along an element, introduced by Mary in \cite{Mary} and to the class of outer generalized inverses introduced by Drazin in \cite{Drazin}.

In the last section we indicate how the concepts from previous sections can be extended to the case of Hilbert space operators.

In a sequel we give some preliminaries.
If $a\in R$ and there exists $x\in R$ such that $axa=a$ then we say that $a$ is von Neumann regular (regular for short) and $x$ is inner generalized inverse of $a$. If $y\in R$ and  $yay=y$ than $y$ is called outer generalized inverse of $a$. An element $x$ is called reflexive generalized inverse of $a$ if $x$ is both inner and outer generalized inverse of $a$. If $x$ satisfies equations $q_1,q_2,...,q_n$ then $x$ is called $\{q_1,q_2,...,q_n\}$ inverse of $a$. The set of all such inverses is denoted by $a\{q_1,q_2,...,q_n\}$. For example, $a\{1,2,5\}=\{a^\#\}$.
We write $R^{(1)}$, $R^\#$, $R^\dag$, $R^\core$, $R_\core$ for the set of all regular, group, MP, core, dual core invertible elements of a ring $R$ respectively. An element $a\in R$ is EP if $a^\#$ and $a^\dag$ exist and $a^\#=a^\dag$. We will denote by $aR$ and $Ra$ the right and left ideal generated by $a$; $aR=\{ax : x\in R\}$ and $Ra=\{xa : x\in R\}$. Also $aRb=\{axb : x\in R\}$. The right annihilator of $a$ is denoted by $a^\circ$ and is defined by $a^\circ = \{x\in R : ax=0\}$. Similarly, the left annihilator of $a$ is the set ${^\circ a}=\{x\in R : xa=0\}$. Finally, if $p,q\in R$ are idempotents
then arbitrary $x\in R$ can be written as $$x=pxq+px(1-q)+(1-p)xq+(1-p)x(1-q)$$ or in the matrix form
$$x=\bmatrix x_{1,1} & x_{1,2} \\ x_{2,1} & x_{2,2}\endbmatrix_{p\times q},$$ where $x_{1,1}=pxq$, $x_{1,2}=px(1-q)$, $x_{2,1}=(1-p)xq$, $x_{2,2}=(1-p)x(1-q)$.
If $x=(x_{i,j})_{p\times q}$ and $y=(y_{i,j})_{p\times q},$
then $x+y=(x_{i,j}+y_{i,j})_{p\times q}$. Moreover, if $r\in R$ is
idempotent and $z=(z_{i,j})_{q\times r}$, then one can use usual matrix rules in order to multiply $x$ and $z$.

\section{Equivalent definitions and properties of $a^\#$, $a^\dag$, $a^\core$ and $a_\core$}

In this section we will give several characterizations for group, MP, core and dual core inverse and obtain some properties. We note that the results stated in theorems \ref{ekvivalentna def za grup inverz}--\ref{idempotenti za dual core inverz} are new even in the case $R=M_n$.

First we show that considered inverses are reflexive generalized inverses with prescribed range and null space.
It is known that $A^\dag$ is reflexive generalized inverse of $A$ with range $\R(A^*)$ and null space $\N(A^*)$, \cite{BIG}. We write $$A^\dag =A^{(1,2)}_{\R(A^*),\N(A^*)}.$$ Also  \cite{BIG}, $$A^\#=A^{(1,2)}_{\R(A),\N(A)}.$$ To find a similar expression for core inverse, recall that $A^\core=A^\#P_A$, \cite{BT}. This means
\begin{equation}\label{core izmedju grup i MP}
  A^\core=A^\#AA^\dag,
\end{equation}
so we obtain
\begin{align*}
  &\R(A)=\R(A^\#)=\R(A^\#AA^\dag AA^\#)\subseteq \R(A^\# AA^\dag)=\R(A^\core)\subseteq \R(A^\#) \\
  &\N(A^*)=\N(A^\dag)=\N(A^\dag AA^\# AA^\dag)\supseteq \N(A^\# AA^\dag)=\N(A^\core)\supseteq \N(A^\dag).
\end{align*}
We see at once that $$A^\core=A^{(1,2)}_{\R(A),\N(A^*)}.$$ Similarly,
$$A_\core=A^{(1,2)}_{\R(A^*),\N(A)}.$$

The definition of $A^\core$ given in Definition \ref{core matrix def} does not make sense in rings. So, we need an equivalent definition.
\begin{lemma}\label{lema alternativna def za core inverz}
  A matrix $X\in M_n$ is the core inverse of $A\in M_n$ if and only if $AXA=A$, $XM_n=AM_n$ and $M_nX=M_nA^*$.
\end{lemma}
\begin{proof}
  Suppose that $X$ is the core inverse of $A$. It is clear that $XM_n\subseteq AM_n$ since $\R(X)\subseteq\R(A)$. By (\ref{core izmedju grup i MP}), we see that $AXA=A$ and $XA=A^\#A$, so $A=XA^2$, hence $AM_n\subseteq XM_n$. Also, $A^*=A^*(AX)^*=A^*AX$ so $M_nA^*\subseteq M_nX$. Finally, $X=A^\#AA^\dag =A^\#(A^\dag)^*A^*$ implies $M_nX\subseteq M_nA^*$. Conversely, suppose that $A=AXA$, $XM_n=AM_n$ and $M_nX=M_nA^*$. It follows that $\R(X)\subseteq \R(A)$ and there exist $V\in M_n$ such that $X=VA^*$. It is now clear that $(AX)^2=AX$, and $X=VA^*=VA^*X^*A^*=XX^*A^*$. Therefore $AX=AX(AX)^*$ which is self-adjoint, so $AX=P_A$.
\end{proof}

Similarly, we can show the analogous result for dual core inverse.

\begin{lemma}\label{lema alternativna def za dual core inverz}
  A matrix $X\in M_n$ is the dual core inverse of $A\in M_n$ if and only if $AXA=A$, $XM_n=A^*M_n$ and $M_nX=M_nA$.
\end{lemma}

Now, we can give the extensions of the concepts of the core and dual core inverse from $M_n$ to $R$.

\begin{definition}\label{definicija core inverse}
  Let $a\in R$. An element $a^\core\in R$ satisfying
  $$aa^\core a=a, \quad a^\core R=aR \quad \text{and} \quad Ra^\core =Ra^*$$
  is called core inverse of $a$.
\end{definition}

\begin{definition}\label{definicija dual core inverse}
  Let $a\in R$. An element $a_\core\in R$ satisfying
  $$aa_\core a=a, \quad a_\core R=a^*R \quad \text{and} \quad Ra_\core =Ra$$
  is called dual core inverse of $a$.
\end{definition}

In the similar way we can give the characterizations of the group and MP inverse. First we need some auxiliary lemmas.

\begin{lemma}\label{podskupovi 1}
Let $a,b\in R$. Then:
\begin{enumerate}[{\rm (i)}]
  \item If $aR\subseteq bR$ then $^\circ b\subseteq {^\circ a}$.
  \item If $b\in R^{(1)}$ and $^\circ b\subseteq {^\circ a}$ then $aR\subseteq bR$.
\end{enumerate}
\end{lemma}
\begin{proof}
  (i): Suppose that $aR\subseteq bR$ and $ub=0$ for some $u\in R$. There exists $x\in R$ such that $a=bx$ so $ua=ubx=0$.

  (ii): Suppose now that $^\circ b\subseteq {^\circ a}$ and $b^{(1)}\in b\{1\}$. Since $(1-bb^{(1)})b=0$ we have $(1-bb^{(1)})a=0$ so $a=bb^{(1)}a$. Therefore, $aR\subseteq bR$.
\end{proof}

\begin{lemma}\label{podskupovi 2}
Let $a,b\in R$. Then:
\begin{enumerate}[{\rm (i)}]
  \item If $Ra\subseteq Rb$ then $b^\circ\subseteq a^\circ$.
  \item If $b\in R^{(1)}$ and $b^\circ\subseteq a^\circ$ then $Ra\subseteq Rb$.
\end{enumerate}
\end{lemma}

\begin{theorem}\label{ekvivalentna def za grup inverz}
  Let $a,x\in R$. The following statements are equivalent:
  \begin{enumerate}[{\rm(i)}]
    \item $a$ is group invertible and $x=a^\#$.
    \item $axa=a$, $xR=aR$ and $Rx=Ra$.
    \item $axa=a$, ${^\circ x}={^\circ a}$ and $x^\circ =a^\circ$.
    \item $axa=a$, $xR\subseteq aR$ and $Rx\subseteq Ra$.
    \item $axa=a$, $^\circ a\subseteq {^\circ x}$ and $a^\circ\subseteq x^\circ$.
  \end{enumerate}
\end{theorem}

\begin{proof}
  (i) $\Rightarrow$ (ii): We have $a=axa=aax=xaa$ and $x=xax=xxa=axx$ so $xR=aR$ and $Rx=Ra$.

  (ii) $\Rightarrow$ (iii) $\Rightarrow$ (iv) $\Rightarrow$ (v) follows by lemmas \ref{podskupovi 1} and \ref{podskupovi 2}.

  (v) $\Rightarrow$ (i): From $axa=a$ it follows that $ax-1\in {^\circ a}\subseteq {^\circ x}$ and $1-xa\in a^\circ\subseteq x^\circ$ so $(ax-1)x=0$ and $x(1-xa)=0$. Now, $x=ax^2=x^2a$, hence $ax=ax^2a=xa$ and $xax=x^2a=x$. By the uniqueness of the group inverse, $x=a^\#$.
\end{proof}

\begin{theorem}\label{ekvivalentna def za MP inverz}
  Let $a,x\in R$. The following statements are equivalent:
  \begin{enumerate}[{\rm(i)}]
    \item $a$ is MP invertible and $x=a^\dag$.
    \item $axa=a$, $xR=a^*R$ and $Rx=Ra^*$.
    \item $axa=a$, ${^\circ x}={^\circ (a^*)}$ and $x^\circ =(a^*)^\circ$.
    \item $axa=a$, $xR\subseteq a^*R$ and $Rx\subseteq Ra^*$.
    \item $axa=a$, ${^\circ (a^*)}\subseteq {^\circ x}$ and $(a^*)^\circ\subseteq x^\circ$.
  \end{enumerate}
\end{theorem}

\begin{proof}
  (i) $\Rightarrow$ (ii): By the properties of MP inverse we easily obtain
   $a^*=xaa^*=a^*ax$ and $x=a^*x^*x=xx^*a^*$ so $xR=a^*R$ and $Rx=Ra^*$.

  (ii) $\Rightarrow$ (iii) $\Rightarrow$ (iv) $\Rightarrow$ (v) follows by lemmas \ref{podskupovi 1} and \ref{podskupovi 2}.

  (v) $\Rightarrow$ (i): Since $a^*x^*a^*=a^*$, we see that $(1-x^*a^*)\in (a^*)^\circ\subseteq x^\circ$ and $(1-a^*x^*)\in {^\circ (a^*)}\subseteq {^\circ x}$. Therefore, $x=xx^*a^*$ and $x=a^*x^*x$. This yields $ax=ax(ax)^*$ and $xa=(xa)^*xa$; hence $ax$ and $xa$ are self-adjoint. Finally, $xax=x(ax)^*=xx^*a^*=x$. It follows that $x=a^\dag$.
\end{proof}
Definitions \ref{definicija core inverse}, \ref{definicija dual core inverse} and theorems \ref{ekvivalentna def za grup inverz} (ii), \ref{ekvivalentna def za MP inverz} (ii) show that group, MP, core and dual core inverses can be defined analogously:
\begin{equation}\label{sve def na jednom mestu}
\begin{aligned}
& x\in R \text{ is group inverse of } a \text{ if and only if } axa=a, \, xR=aR, \, Rx=Ra,\\
& x\in R \text{ is MP inverse of } a \text{ if and only if } axa=a, \, xR=a^*R, \, Rx=Ra^*,\\
& x\in R \text{ is core inverse of } a \text{ if and only if } axa=a, \, xR=aR, \, Rx=Ra^*\\
& x\in R \text{ is dual core inverse of } a \text{ if and only if } axa=a, \, xR=a^*R, \, Rx=Ra.
\end{aligned}
\end{equation}
As we can see, the four inverses are closely related and it can be said that they form a certain subclass of the class of all inner inverses. Moreover, we can conclude that core and dual core inverse are between group and MP inverse.

We will now show that the existence of considered inverses is closely related with existence of some idempotents. First, we give some auxiliary results.

\begin{lemma}\label{lema jedinstveno q}
  If $q_1$ and $q_2$ are idempotents such that $Rq_1\subseteq Rq_2$ and $q_2R\subseteq q_1R$ then $q_1=q_2$.
\end{lemma}
\begin{proof}
  If $Rq_1\subseteq Rq_2$ then $q_1=uq_2$ for some $u\in R$ so $q_1q_2=uq_2^2=uq_2=q_1$. Similarly, $q_2R\subseteq q_1R$ implies $q_1q_2=q_2$.
\end{proof}

\begin{lemma}\label{lema jedinstveni p i r}
  If $p_1$ and $p_2$ are self-adjoint idempotents such that $Rp_1=Rp_2$ or $p_1R=p_2R$ then $p_1=p_2$.
\end{lemma}
\begin{proof}
  If $Rp_1=Rp_2$ then, like in previous lemma, $p_1=p_1p_2$ and $p_2=p_2p_1$. But $p_2=p_2^*=p_1^*p_2^*=p_1p_2=p_1$. Similarly, $p_1R=p_2R$ implies $p_1=p_2$.
\end{proof}

\begin{theorem}\label{idempotenti za grup inverz}
  Let $a\in R$. The following assertions are equivalent:
  \begin{enumerate}[{\rm (i)}]
    \item $a$ is group invertible.
    \item There exists an idempotent $q\in R$ such that $qR=aR$ and $Rq=Ra$.
    \item $a\in R^{(1)}$ and there exists idempotent $q\in R$ such that ${^\circ a}={^\circ q}$ and $a^\circ=q^\circ$.
  \end{enumerate}
  If the previous assertions are valid then the assertions (ii) and (iii) deal with the same unique idempotent $q$. Moreover, $qa^{(1)}q$ is invariant under the choice of $a^{(1)}\in a\{1\}$ and
  \begin{equation}\label{matricne forme za grup inverz}
    a=\bmatrix a & 0 \\ 0 & 0 \endbmatrix_{q\times q}, \quad a^\#=\bmatrix qa^{(1)}q & 0 \\ 0 & 0 \endbmatrix_{q\times q}.
  \end{equation}
\end{theorem}
\begin{proof}
(i) $\Rightarrow$ (ii):  Suppose that $a$ is group invertible and set $q=aa^\#=a^\#a$. Then $a=qa=aq$ so $qR=aR$, $Rq=Ra$.

(ii) $\Rightarrow$ (iii): From $qR=aR$ we have $q=ax$ and $a=qz$ for some $x,z\in R$. Therefore, $qa=q^2z=qz=a$ and $axa=qa=a$, so $a\in R^{(1)}$. The rest of the proof follows by Lemma \ref{podskupovi 1} (i) and Lemma \ref{podskupovi 2} (i).

(iii) $\Rightarrow$ (i):
   Suppose that $a\in R^{(1)}$ and suppose that there exists an idempotent $q$ such that $a^\circ=q^\circ$ and ${^\circ a}={^\circ q}$. Let $a^{(1)}\in a\{1\}$ be arbitrary. Since $1-a^{(1)}a\in a^\circ \subseteq q^\circ$ we obtain $q=qa^{(1)}a$. Also, $1-q\in q^\circ\subseteq a^\circ$, so $a=aq$. Similarly, $q=aa^{(1)}q$ and $a=qa$. Set $x=qa^{(1)}q$. We have $x=a^\#$, because
   \begin{align*}
   &ax=aqa^{(1)}q=aa^{(1)}q=q, \quad xa=qa^{(1)}qa=qa^{(1)}a=q,\\
   &axa=qa=a, \quad xax=qx=x.
   \end{align*}
   Now the invariance of $qa^{(1)}q$ under the choice of  $a^{(1)}\in a\{1\}$ follows. Note that we have also proved representations (\ref{matricne forme za grup inverz}) since $a=qaq$ and $a^\#=qa^{(1)}q$. The uniqueness of $q$ follows by Lemma \ref{lema jedinstveno q}.
\end{proof}

\begin{theorem}\label{idempotenti za MP inverz}
  Let $a\in R$. The following assertions are equivalent:
  \begin{enumerate}[{\rm (i)}]
    \item $a$ is MP invertible.
    \item There exist self-adjoint idempotents $p,r\in R$ such that $pR=aR$ and $Rr=Ra$.
    \item $a\in R^{(1)}$ and there exist self-adjoint idempotents $p,r\in R$ such that ${^\circ a}={^\circ p}$ and $a^\circ=r^\circ$.
  \end{enumerate}
  If the previous assertions are valid then the assertions (ii) and (iii) deal with the same pair of unique self-adjoint idempotents $p$ and $r$. Moreover, $ra^{(1)}p$ is invariant under the choice of $a^{(1)}\in a\{1\}$ and
  \begin{equation}\label{matricne forme za MP inverz}
    a=\bmatrix a & 0 \\ 0 & 0 \endbmatrix_{p\times r}, \quad a^\dag=\bmatrix ra^{(1)}p & 0 \\ 0 & 0 \endbmatrix_{r\times p}.
  \end{equation}
\end{theorem}
\begin{proof}
(i) $\Rightarrow$ (ii):
  Suppose that $a$ is MP invertible and set $p=aa^\dag$ and $r=a^\dag a$. It is clear that $p$ and $r$ are self-adjoint idempotents. Since $a=pa=ar$ we conclude that $pR=aR$ and $Rr=Ra$.

(ii) $\Rightarrow$ (iii): If we use $p$ instead of $q$ then the proof proceeds along the same lines as the proof of Theorem \ref{idempotenti za grup inverz} (ii) $\Rightarrow$ (iii).

(iii) $\Rightarrow$ (i): As in the proof of Theorem \ref{idempotenti za grup inverz} we can show that $a=pa=ar$, $p=aa^{(1)}p$ and $r=ra^{(1)}a$. Set $x=ra^{(1)}p$. We have $x=a^\dag$ because
\begin{align*}
  & ax=ara^{(1)}p=aa^{(1)}p=p=p^* \\
  & xa=ra^{(1)}pa=ra^{(1)}a=r=r^* \\
  & axa=pa=a \\
  & xax=rx=x.
\end{align*}
Now the invariance of $ra^{(1)}p$ under the choice of  $a^{(1)}\in a\{1\}$ follows because it is known that MP inverse is unique when it exists. Note that we have also proved representations (\ref{matricne forme za MP inverz}) since $a=par$ and $a^\dag =x=ra^{(1)}p$. The uniqueness of $p$ and $r$ follows by Lemma \ref{lema jedinstveni p i r}.
\end{proof}

Recall that a $*$-ring $R$ is Rickart $*$-ring if for every $a\in R$ there exists self-adjoint idempotent $p$ such that $^\circ a=Rp$, \cite{Be}. The analogous property for right annihilators is automatically fulfilled in this case. Note that $Rp={^\circ (1-p)}$.

\begin{corollary}
  Let $a\in R$ where $R$ is Rickart $*$-ring. Then $a$ is MP invertible if and only if $a$ is regular.
\end{corollary}

The analogous characterizations of core and dual core inverses using idempotents and annihilators are given in next two theorems. Furthermore, we characterize these inverses by the set of equations.
\begin{theorem}\label{idempotenti za core inverz}
  Let $a\in R$. The following assertions are equivalent:
  \begin{enumerate}[{\rm (i)}]
    \item $a$ is core invertible.
    \item There exists $x\in R$ such that $axa=a$, ${^\circ x}={^\circ a}$ and $x^\circ=(a^*)^\circ$.
    \item There exist $x\in R$ such that $$(1)\; axa=a \quad (2)\; xax=x \quad (3)\; (ax)^*=ax \quad (6)\; xa^2=a \quad (7)\; ax^2=x.$$
    \item There exist self-adjoint idempotent $p\in R$ and idempotent $q\in R$ such that $pR=aR$, $qR=aR$ and $Rq=Ra$.
    \item $a\in R^{(1)}$ and there exist self-adjoint idempotent $p$ and idempotent $q\in R$ such that ${^\circ a}={^\circ p}$, ${^\circ a}={^\circ q}$ and $a^\circ=q^\circ$.
  \end{enumerate}
  If the previous assertions are valid then $x=a^\core$, $a^\core$ is unique and the assertions (iv) and (v) deal with the same pair of unique idempotents $p$ and $q$. Moreover, $qa^{(1)}p$ is invariant under the choice of $a^{(1)}\in a\{1\}$ and
  \begin{equation}\label{matricne forme za core inverz}
    a=\bmatrix a & 0 \\ 0 & 0 \endbmatrix_{p\times q}, \quad a^\core=\bmatrix qa^{(1)}p & 0 \\ 0 & 0 \endbmatrix_{q\times p}.
  \end{equation}
\end{theorem}
\begin{proof}
  (i) $\Rightarrow$ (ii):
  Suppose that $a$ is core invertible and let $x=a^\core$. By definition, $axa=a$, $xR=aR$ and $Rx=Ra^*$. By lemmas \ref{podskupovi 1} and \ref{podskupovi 2}, it follows that ${^\circ x}={^\circ a}$ and $x^\circ=(a^*)^\circ$.

  (ii) $\Rightarrow$ (iii)
  Suppose that there exists $x\in R$ such that $axa=a$, ${^\circ x}={^\circ a}$ and $x^\circ=(a^*)^\circ$. We can follow the proofs of theorems \ref{ekvivalentna def za grup inverz} and \ref{ekvivalentna def za MP inverz} to obtain that
  \begin{equation*}
    x=ax^2,\quad ax=(ax)^* \quad \text{and} \quad xax=x.
  \end{equation*}
  From $(xa-1)x\in {^\circ x}\subseteq{^\circ a}$ we have
  \begin{equation*}
    a=xa^2.
  \end{equation*}

  (iii) $\Rightarrow$ (iv): Set $p=ax$ and $q=xa$. From $axa=a$ it follows that $p$ and $q$ are idempotents such that $pR=aR$ and $Rq=Ra$. Equation (3) shows that $p$ is self-adjoint. From $a=xa^2=qa$ and $q=xa=ax^2a$ we conclude that $qR=aR$.

  (iv) $\Rightarrow$ (v): The proof is similar to the proof of Theorem \ref{idempotenti za grup inverz} (ii) $\Rightarrow$ (iii).

  (v) $\Rightarrow$ (i):
  Suppose that $a\in R^{(1)}$ and suppose that there exist self-adjoint idempotent $p\in R$ and idempotent $q\in R$ such that ${^\circ a}={^\circ p}$, ${^\circ a}={^\circ q}$ and $a^\circ=q^\circ$. Fix $a^{(1)}\in a\{1\}$. In the proof of Theorem \ref{idempotenti za grup inverz} we showed that $a=qa=aq$ and $q=qa^{(1)}a=aa^{(1)}q$. In the proof of Theorem \ref{idempotenti za MP inverz} we showed that $a=pa$ and $p=aa^{(1)}p$. Let $a^-\in a\{1\}$ be arbitrary. Then $qa^-p=qa^{(1)}aa^-aa^{(1)}p=qa^{(1)}aa^{(1)}p=qa^{(1)}p$, so $qa^-p$ is invariant under the choice of $a^-\in a\{1\}$. Set $x=qa^{(1)}p$. We have $axa=aqa^{(1)}pa=aa^{(1)}a=a$. Also, $x=qa^{(1)}p=aa^{(1)}qa^{(1)}p$ and $xa^2=qa^{(1)}pa^2=qa^{(1)}aa=qa=a$, so $xR=aR$. Moreover,
  \begin{align*}
    &x=qa^{(1)}p^*=qa^{(1)}(aa^{(1)}p)^*=qa^{(1)}p(a^{(1)})^*a^* \quad \text{and}\\
    &a^*ax=a^*aqa^{(1)}p=a^*aa^{(1)}p=a^*p=(pa)^*=a^*,
  \end{align*}
  so $Rx=Ra^*$. It follows that $x=a^\core$, i.e. $a$ is core invertible.

  The uniqueness of $p$ and $q$ follows by lemmas \ref{lema jedinstveni p i r} and \ref{lema jedinstveno q}.
  If $x$ is core inverse of $a$ then we showed that $x$ has properties given in (ii) and (iii). Suppose that there exist two elements $x$ and $y$ satisfying equations in (iii). By the proof of (iii) $\Rightarrow$ (iv) and the uniqueness of $p$ and $q$ we conclude that $p=ax=ay$ and $q=xa=ya$.
  Therefore, $x=xax=yay=y$. We also proved that if exists some $x$ satisfying equations in (iii) then $a$ is core invertible but its core inverse must satisfies equations in (iii) which uniquely determine $x$. It follows that $x$ appearing in (ii) and $x$ appearing in (iii) are both equal to $a^\core$ and that core inverse of $a$ is unique.
  Representations (\ref{matricne forme za core inverz}) follows by $a=paq$ and $a^\core =x=qa^{(1)}p$.
\end{proof}

The theorem concerning the dual core inverse can be proved similarly.
\begin{theorem}\label{idempotenti za dual core inverz}
  Let $a\in R$. The following assertions are equivalent:
  \begin{enumerate}[{\rm (i)}]
    \item $a$ is dual core invertible.
    \item There exists $x\in R$ such that $axa=a$, ${^\circ x}={^\circ (a^*)}$ and $x^\circ=a^\circ$.
    \item There exists $x\in R$ such that $$(1)\; axa=a \quad (2)\; xax=x \quad (4)\; (xa)^*=xa \quad (8)\; a^2x=a \quad (9)\; x^2a=x.$$
    \item There exist self-adjoint idempotent $r\in R$ and idempotent $q\in R$ such that $Rr=Ra$, $qR=aR$ and $Rq=Ra$.
    \item $a\in R^{(1)}$ and there exist self-adjoint idempotent $r$ and idempotent $q\in R$ such that $a^\circ=r^\circ$, ${^\circ a}={^\circ q}$ and $a^\circ=q^\circ$.
  \end{enumerate}
  If the previous assertions are valid then $x=a_\core$, $a_\core$ is unique and the assertions (iv) and (v) deal with the same pair of unique idempotents $r$ and $q$. Moreover, $ra^{(1)}q$ is invariant under the choice of $a^{(1)}\in a\{1\}$ and
  \begin{equation*}
    a=\bmatrix a & 0 \\ 0 & 0 \endbmatrix_{q\times r}, \quad a_\core=\bmatrix ra^{(1)}q & 0 \\ 0 & 0 \endbmatrix_{r\times q}.
  \end{equation*}
\end{theorem}

We will use $q_a$, $p_a$ and $r_a$ for the idempotents associated with $a\in R$, given in theorems \ref{idempotenti za grup inverz} -- \ref{idempotenti za dual core inverz}. We will write $q$, $p$ and $r$ instead of $q_a$, $p_a$ and $r_a$ when no confusion can arise.

Let us look at the equations in theorems \ref{idempotenti za core inverz} (iii) and \ref{idempotenti za dual core inverz} (iii) that characterize core and dual core inverse respectively. Note that these equations are combinations of equations that characterize group inverse and equations that characterize MP inverse. To see that it is enough to check that the following sets of equations are equivalent:
\begin{enumerate}[{\rm (i)}]
\item $axa=a$, $xax=x$, $ax=xa$;
\item $axa=a$, $xax=x$, $xa^2=a$, $a^2x=a$;
\item $axa=a$, $xax=x$, $x^2a=x$, $ax^2=x$.
\end{enumerate}
It is clear that (i) implies (ii) and (iii). If $axa=a$, $xax=x$, $xa^2=a$, $a^2x=a$ then $ax=xa^2x=xa$, so (ii) implies (i). Similarly, (iii) implies (i).
\begin{remark}
From theorems \ref{idempotenti za grup inverz} -- \ref{idempotenti za dual core inverz} it follows that $a$ is both group and MP invertible if and only if $a$ is both core and dual core invertible. If $a$ is core or dual core invertible then $a$ is group invertible. In other words, $R^\#\cap R^\dag =R^\core \cap R_\core$ and $R^\core \cup R_\core\subseteq R^\#$. But, core invertibility or dual core invertibility of $a$ does not imply MP invertibility of $a$.
\end{remark}

\begin{remark}
  The statements (ii) and (iii) in Theorem \ref{idempotenti za core inverz} and the statements (ii) and (iii) of Theorem \ref{idempotenti za dual core inverz} can be used as equivalent definitions of core inverse and dual core inverse, respectively.
\end{remark}

Suppose that $a\in R^\#\cap R^\dag$. By theorems \ref{idempotenti za grup inverz} and \ref{idempotenti za MP inverz}, it follows that there exist unique idempotent $q=q_a$ and unique self-adjoint idempotents $p=p_a$ and $r=r_a$ with given properties. By the uniqueness, we conclude that these idempotents are the same as idempotents in theorems \ref{idempotenti za core inverz} and \ref{idempotenti za dual core inverz}. Therefore,
\begin{equation}\label{p,q,r}
\begin{aligned}
  &q=aa^\#=a^\#a=a^\core a=aa_\core\\
  & p=aa^\dag =aa^\core \\
  & r=a^\dag a=a_\core a.
\end{aligned}
\end{equation}
Now, it is easy to show that
\begin{equation}\label{pq=q}
  pq=q, \quad qp=p,\quad rq=r, \quad qr=q.
\end{equation}
Moreover,
\begin{equation}\label{pq*=p}
  q^*p=(pq)^*=q^*, \quad pq^*=(qp)^*=p, \quad q^*r=(rq)^*=r, \quad rq^*=(qr)^*=q^*.
\end{equation}
We also proved in theorems \ref{idempotenti za grup inverz} -- \ref{idempotenti za dual core inverz} that
\begin{equation}\label{inverzi na jednom mestu}
  a=qaq=paq=qar=par, \quad a^\#=qa^{(1)}q, \quad a^\dag=ra^{(1)}p, \quad a^\core =qa^{(1)}p, \quad a_\core =ra^{(1)}q,
\end{equation}
where $a^{(1)}\in a\{1\}$ is arbitrary.
By (\ref{pq=q}) -- (\ref{inverzi na jednom mestu}), it follows that
\begin{equation}\label{matricne forme prosirene}
\begin{aligned}
  a&=\bmatrix a&0 \\ 0&0 \endbmatrix_{q\times q} =\bmatrix a&0 \\ 0&0 \endbmatrix_{p\times q}=\bmatrix a&0 \\ 0&0 \endbmatrix_{q\times r}=\bmatrix a&0 \\ 0&0 \endbmatrix_{p\times r} \\
  a^\# &=\bmatrix a^\# &0 \\ 0&0 \endbmatrix_{q\times q}=\bmatrix a^\#&0 \\ 0&0 \endbmatrix_{p\times q}=\bmatrix a^\#&0 \\ 0&0 \endbmatrix_{q\times r}=\bmatrix a^\#&0 \\ 0&0 \endbmatrix_{p\times r} \\
  a^\dag &=\bmatrix a^\dag &0 \\ 0&0 \endbmatrix_{r\times p}=\bmatrix a^\dag &0 \\ 0&0 \endbmatrix_{q^*\times p}=\bmatrix a^\dag &0 \\ 0&0 \endbmatrix_{r\times q^*}=\bmatrix a^\dag &0 \\ 0&0 \endbmatrix_{q^*\times q^*} \\
  a^\core &=\bmatrix a^\core &0 \\ 0&0 \endbmatrix_{q\times p}=\bmatrix a^\core &0 \\ 0&0 \endbmatrix_{p\times p}=\bmatrix a^\core &0 \\ 0&0 \endbmatrix_{q\times q^*}=\bmatrix a^\core &0 \\ 0&0 \endbmatrix_{p\times q^*}\\
  a_\core &=\bmatrix a_\core &0 \\ 0&0 \endbmatrix_{r\times q}=\bmatrix a_\core &0 \\ 0&0 \endbmatrix_{q^*\times q}=\bmatrix a_\core &0 \\ 0&0 \endbmatrix_{r\times r}=\bmatrix a_\core &0 \\ 0&0 \endbmatrix_{q^*\times r}.
\end{aligned}
\end{equation}
The elements in upper left corners in (\ref{matricne forme prosirene}) belong to the sets of the forms $p_1Rp_2$, where $p_1$ and $p_2$ are idempotents. When $p_1\neq p_2$ we can not consider the invertibility of the corner element in $p_1Rp_2$, but it has some similar property. Let us look, for example, the representation $a^\core= \bmatrix a^\core &0 \\ 0&0 \endbmatrix_{q\times q^*}\in qRq^*$. There exists unique element $x\in q^*Rq$ such that $xa^\core=q^*$ and $a^\core x=q$. Namely, $x=q^*aq$. The analogous property can be shown for all corner elements in (\ref{matricne forme prosirene}). The proof is left to the reader. We will back to this property in Section 4 when we will consider the case of Hilbert space operators.

It is clear that $(a^\#)^\#=a$ and $(a^\dag)^\dag=a$. The expressions for $(A^\core)^\dag$ and $(A^\core)^\core$, where $A\in M_n$, are given in \cite{BT}. We give expressions for other "double" inverses.
\begin{theorem}\label{dupli inverzi}
  Let $a\in R^\#\cap R^\dag$. Then:
  \begin{enumerate}[{\rm (i)}]
    \item $p_{a^\#}=p_a$,\quad $q_{a^\#}=q_a$,\quad $r_{a^\#}=r_a$ and \begin{flalign*}(a^\#)^\#=a, \quad (a^\#)^\dag=r_aap_a, \quad (a^\#)^\core=ap_a, \quad (a^\#)_\core=r_aa.&&\end{flalign*}
    \item $p_{a^\dag}=r_a$,\quad $q_{a^\dag}=q_a^*$,\quad $r_{a^\dag}=p_a$ and \begin{flalign*}(a^\dag)^\#=q^*_aaq^*_a, \quad (a^\dag)^\dag=a, \quad (a^\dag)^\core=q^*_aa, \quad (a^\dag)_\core=aq^*_a.&&\end{flalign*}
    \item $p_{a^\core}=q_{a^\core}=r_{a^\core}=p_a$ and \begin{flalign*}(a^\core)^\#=(a^\core)^\dag=(a^\core)^\core=(a^\core)_\core=ap_a.&&\end{flalign*}
    \item $p_{a_\core}=q_{a_\core}=r_{a_\core}=r_a$ and
      \begin{flalign*}(a_\core)^\#=(a_\core)^\dag=(a_\core)^\core=(a_\core)_\core=r_aa.&&\end{flalign*}
  \end{enumerate}
\end{theorem}
\begin{proof}
  We give the proof only for the statement (iii); the other statements may be proved in the same manner. Since $a^\core R=aR=p_aR$ and $Ra^\core=Ra^*=(aR)^*=(p_aR)^*=Rp_a$, we conclude that $$p_{a^\core}=q_{a^\core}=r_{a^\core}=p_a.$$ By (\ref{inverzi na jednom mestu}), we obtain
      \begin{align*}
        (a^\core)^\#&=(a^\core)^\dag=(a^\core)^\core =(a^\core)_\core=\bmatrix p_a(a^\core)^{(1)}p_a & 0 \\ 0 & 0 \endbmatrix_{p_a\times p_a}\\
        &=\bmatrix p_aap_a & 0 \\ 0 & 0 \endbmatrix_{p_a\times p_a}=\bmatrix ap_a & 0 \\ 0 & 0 \endbmatrix_{p_a\times p_a}.
      \end{align*}
\end{proof}

From Theorem \ref{dupli inverzi} it follows that $a^\core$ and $a_\core$ are EP.
The properties of core inverse given in the following theorem is a generalization of the case $R=M_n$ (see \cite{BT}) to the case of arbitrary $*$-ring.

\begin{theorem}\label{generalizacija}
  Let $a\in R^\core$ and $n\in \mathbb{N}$. Then:
  \begin{enumerate}[{\rm (i)}]
    \item $a^\core=a^\#p_a$;
    \item $(a^\core)^2a=a^\#$;
    \item $(a^\core)^n=(a^n)^\core$;
    \item $((a^\core)^\core)^\core=a^\core$;
    \item If $a\in R^\dag$ then $$a^\#=a^\core aa_\core, \quad a^\dag=a_\core aa^\core, \quad a^\core=a^\#aa^\dag, \quad a_\core=a^\dag aa^\#.$$
  \end{enumerate}
\end{theorem}
\begin{proof}
Since $a\in R^\core$ we have the existence of $a^\#$, $q_a$ and $p_a$.
\begin{enumerate}[{\rm (i):}]
  \item By Theorem \ref{idempotenti za core inverz} (iii) and (\ref{p,q,r}), we have $a^\core=a^\core aa^\core =a^\#aa^\core =a^\# p_a$.
  \item $(a^\core)^2a=a^\core q_a\stackrel{(i)}{=}a^\#p_aq_a\stackrel{(\ref{pq=q})}{=}a^\#q_a=a^\#$.
  \item Since $a=a^n(a^\#)^{n-1}=(a^\#)^{n-1}a^n$ we conclude that $Ra^n=Ra=Rq_a$ and $a^nR=aR=q_aR=p_aR$. Using $a(a^\core)^2=a^\core$ we see that $a^n(a^\core)^n=aa^\core$ so $a^n(a^\core)^na^n=aa^\core a^n=a^n$, i.e. $(a^\core)^n\in a^n\{1\}$. By Theorem \ref{idempotenti za core inverz}, we obtain
      $$(a^n)^\core=q_a(a^n)^{(1)}p_a=q_a(a^\core)^np_a=(a^\core)^n.$$
  \item By the proof of (iii) of Theorem \ref{dupli inverzi}, we obtain $$((a^\core)^\core)^\core=a^\core p_{a^\core}=a^\core p_a=a^\core.$$
  \item If $a\in R^\dag$ then, by (\ref{p,q,r}), $a^\#=a^\# aa^\#=a^\core aa_\core$, $a^\dag=a^\dag aa^\dag=a_\core aa^\core$, $a^\core=a^\core aa^\core=a^\#aa^\dag$ and $a_\core=a_\core aa_\core=a^\dag aa^\#$.
\end{enumerate}
\end{proof}
The analogous result for dual core inverse of $a\in R_\core$ is valid. The expressions in (v) in Theorem \ref{generalizacija} perhaps best illustrate the connection between the group, MP, core and dual core inverse. Once again, we see that the core and dual core inverse are between group and MP inverse and vice versa.

\section{Characterizations of EP elements}

In this section, we consider the equivalent conditions for EP-ness of $a\in R$.
Recall that $a\in R$ is EP if $a\in R^\# \cap R^\dag$ and $a^\#=a^\dag$. EP matrices and EP operators have been extensively studied. Recently, the EP elements are investigated in the context of rings with involution. For a recent account of the theory see, for example, \cite{Chen}, \cite{Mosic} and the references given there.

\begin{theorem}\label{theorem EP 1}
  Let $a\in R$. The following assertions are equivalent:
  \begin{enumerate}[{\rm (i)}]
    \item $a$ is EP, i.e. $a\in R^\#\cap R^\dag$ and $a^\#=a^\dag$.
    \item $a\in R^\dag$ and $p_a=r_a$.
    \item $a\in R^\core$ and $p_a=q_a$.
    \item $a\in R_\core$ and $r_a=q_a$.
    \item $a\in R^\core$ and $a^\#=a^\core$.
    \item $a\in R_\core$ and $a^\#=a_\core$.
    \item $a\in R^\#\cap R^\dag$ and $a^\dag=a^\core$.
    \item $a\in R^\#\cap R^\dag$ and $a^\dag=a_\core$.
    \item $a\in R^\#\cap R^\dag$ and $a^\core=a_\core$.
  \end{enumerate}
\end{theorem}
\begin{proof}
  (i) $\Rightarrow$ (ii) -- (ix): If $a$ is EP then $$p_a=aa^\dag=aa^\#=q_a=a^\#a=a^\dag a=r_a.$$ By (\ref{inverzi na jednom mestu}), $$a^\#=a^\dag=a^\core=a_\core=q_aa^{(1)}q_a.$$

  (ii) or (iii) or (iv) $\Rightarrow$ (i): Suppose that $p_a=r_a$. We have $p_aR=aR$ and $Rp_a=Rr_a=Ra$ so there exists $q_a$ and $p_a=q_a=r_a$.  Hence, $a\in R^\#$ and $aa^\dag=p_a=q_a=aa^\#$. Hence, $aa^\#=a^\#a$ is self-adjoint, so $a^\dag=a^\#$.
  Similarly, $p_a=q_a$ or $r_a=q_a$ imply $p_a=r_a=q_a$ and we can proceed as before.

  (v) $\Rightarrow$ (i): Suppose that $a\in R^\core$ and $a^\core=a^\#$. Multiplying both sides by $a$ from the left, we obtain $p_a=aa^\core=aa^\#=q_a$. From the previous part of the proof, it follows that $a$ is EP.

  The remaining implications may be shown similarly.
\end{proof}
Thus, $a$ is EP if and only if $a\in R^\# \cap R^\dag$ and $a^\#=a^\dag=a^\core =a_\core$. Some characterizations in the next theorem involve only group and MP inverse. Note that some of these characterizations are known.
We give them for completeness. For $x,y\in R$ we write $[x,y]=xy-yx$.
\begin{theorem}\label{isprdjivanje}
  Let $a\in R^\dag\cap R^\#$. Then the following assertions are equivalent:
  \begin{enumerate}[{\rm (i)}]
    \item $a$ is EP.
    \item At least one (any) element of the set $$\{[a,a^\dag],[a,a^\core],[a,a_\core],[a^\#,a^\dag],[a^\#,a^\core],[a^\#,a_\core]\}$$
        is equal zero.
    \item At least one (any) element of the set $$\{ap_a,r_aa,r_aap_a,q^*_aa,aq^*_a,q^*_aaq^*_a\}$$
        is equal $a$.
    \item $ap_a=r_aa$.
    \item $r_aap_a=r_aa$.
    \item $r_aap_a=ap_a$.
    \item $q_a^*a=ap_a$.
    \item $aq_a^*=r_aa$.
  \end{enumerate}
\end{theorem}
\begin{proof}
  Write $p=p_a$, $q=q_a$ and $r=r_a$.

  (i) $\Rightarrow$ (ii)--(ix): If $a$ is EP then by Theorem \ref{theorem EP 1}, $a^\#=a^\dag=a^\core=a_\core$ and $q=p=r=q^*$. Now, the proofs easily follow.

  For the proofs of converse implications we use (\ref{p,q,r})--(\ref{inverzi na jednom mestu}) and Theorem \ref{theorem EP 1} or one of the preceding already establish conditions.

  (ii) $\Rightarrow$ (i): We need to show that if there exist some element from the set which is equal zero then $a$ is EP. If $aa^\core=a^\core a$ then $p=q$. By Theorem \ref{theorem EP 1}, $a$ is EP. Suppose that $[a^\#,a^\core]=0$, that is $a^\#a^\core=a^\core a^\#$. Multiplying both sides from the left by $a$ we obtain $qa^\core=pa^\#$. By (\ref{inverzi na jednom mestu}), $a^\core=a^\#$, so $a$ is EP. Suppose that $a^\#a^\dag=a^\dag a^\#$. Multiplying both sides from the left by $a$ we obtain $a^\core=pa^\#=a^\#$, so $a$ is EP. Other cases ($[a,a^\dag]=0$, $[a,a_\core]=0$, $[a^\#,a_\core]=0$) may be proved similarly.

  (iii) $\Rightarrow$ (i): If $ap=a$ then, multiplying both sides from the left by $a^\#$, we obtain $qp=q$, hence, $p=q$. Therefore, $a$ is EP. If $ra=a$ then $q=aa^\#=raa^\#=rq=r$, thus $a$ is EP. If $rap=a$ then $a=qa=qrap=qap=ap$. If $q^*a=a$ then $a=q^*a=rq^*a=ra$. If $aq^*=a$ then $a=aq^*=aq^*p=ap$. Finally, if $q^*aq^*=a$ then $a=q^*aq^*p=ap$.

  (iv) $\Rightarrow$ (i): Suppose that $ap=ra$. Since $qr=q$ we have $a=qa=qra=qap=ap$. By the previous part of the proof, we conclude that $a$ is EP.

  (v) $\Rightarrow$ (i): If $rap=ra$ then $a=qa=qra=qrap=qap=ap$.

  (vi) $\Rightarrow$ (i): If $rap=ap$ then $a=aq=apq=rapq=raq=ra$.

  (vii) $\Rightarrow$ (i): Suppose that $q^*a=ap$. Since $pq^*=p$ we obtain $a=pa=pq^*a=pap=ap$. By the part (iii) $\Rightarrow$ (i) it follows that $a$ is EP.

  (viii) $\Rightarrow$ (i): As $q^*r=r$ we conclude that $aq^*=ra$ implies $a=ar=aq^*r=rar=ra$, so $a$ is EP.
\end{proof}
Combining Theorem \ref{dupli inverzi} with Theorem \ref{isprdjivanje}, we can generalize results from \cite{BT} and obtain a large number of new characterizations of EP-ness of $a$.

\section{Connection with some classes of generalized inverses}
In this section we will show that group, MP, core and dual core inverse belong to some specific classes of generalized inverses.
Recently, Mary introduced in \cite{Mary} a new generalized inverse in semigroup $S$ called the inverse along an element. We consider the case when $S$ is a $*$-ring $R$. For $a,b\in R$, pre-order relation $\mathcal{H}$ is defined in \cite{Mary} by
$$a\leq_\mathcal{H}b \; \Longleftrightarrow \; Ra\subseteq Rb \text{ and } aR\subseteq bR.$$
\begin{definition}\label{inverse along element}(\cite{Mary})
  Let $a,d\in R$. We say that $x\in R$ is an inverse of $a$ along $d$ if it satisfies
$$xad=d=dax \quad \text{and} \quad x\leq_\mathcal{H}d.$$
\end{definition}
It is proved that if an inverse of $a$ along $d$ exists, it is unique and it is outer generalized inverse of $a$. Mary proved in Theorem 11 in \cite{Mary} that $a\in R$ is group invertible if and only if it is invertible along $a$ in which case the inverse of $a$ along $a$ coincides with the group inverse of $a$. Also, $a\in R$ is MP invertible if and only if it is invertible along $a^*$ in which case the inverse of $a$ along $a^*$ coincides with the MP inverse of $a$.

Recently, Drazin independently defined in \cite{Drazin} a new outer generalized inverse in semigroup $S$ that is actually similar to the inverse along an element. We consider the case when $S$ is a $*$-ring $R$.
\begin{definition}\label{bc inverse}(\cite{Drazin})
  Let $a,b,c,x\in R$. Then we shall call $x$ a $(b,c)$-inverse of $a$ if both
  \begin{enumerate}[{\rm (1)}]
    \item $x\in (bRx)\cap (xRc)$ and
    \item $xab=b$, $cax=c$.
  \end{enumerate}
\end{definition}
It is proved that there can be at most one $(b,c)$-inverse $x$ of $a$ and $xax=x$. Drazin proved in \cite{Drazin} that $a^\#$ is $(a,a)$-inverse of $a$ and that $a^\dag$ is $(a^*,a^*)$ inverse of $a$.

Our aim is to connect the core and dual core inverse of $a$ with generalized inverses given in definitions \ref{inverse along element} and \ref{bc inverse}.

\begin{theorem}\label{duz elementa}
  Let $a\in R^\dag$. Then:
  \begin{enumerate}[{\rm (i)}]
    \item $a$ is core invertible if and only if it is invertible along $aa^*$. In this case the inverse along $aa^*$ coincides with core inverse of $a$.
    \item $a$ is dual core invertible if and only if it is invertible along $a^*a$. In this case the inverse along $a^*a$ coincides with dual core inverse of $a$.
  \end{enumerate}
\end{theorem}
\begin{proof}
  (i): Suppose that $a\in R^\dag\cap R^\#$ and let us prove that $x=a^\core$ is inverse of $a$ along $d=aa^*$. Recall that $xa^2=a$ and $(ax)^*=ax$, by Theorem \ref{idempotenti za core inverz}.
  We see at once that
  \begin{align*}
    xad&=x aaa^*=aa^*=d \quad \text{and}\\
    dax&=aa^*ax=aa^*(ax)^*=a(axa)^*=aa^*=d.
  \end{align*}
  We proceed with following observation. Let $z=(a^\dag)^*a^\dag$. Then
  \begin{align}
  aa^*z&=aa^*(a^\dag)^*a^\dag=a(a^\dag a)^*a^\dag=aa^\dag aa^\dag=aa^\dag \quad \text{and} \nonumber\\
  zaa^*&=(a^\dag)^*a^\dag aa^*=(a^\dag)^*(a^\dag a)^*a^*=(aa^\dag aa^\dag)^*=(aa^\dag)^*=aa^\dag. \label{z}
  \end{align}
  Since $ax^2=x$, $xax=x$ and $ax=aa^\dag$ we have
  \begin{align*}
    x&=ax^2=aa^\dag x=aa^*zx=dzx \quad \text{and}\\
    x&=xax=xaa^\dag=xzaa^*=xzd.
  \end{align*}
  It follows that $x\in dR$ and $x\in Rd$ so $xR\subseteq dR$ and $Rx\subseteq Rd$; hence $x\leq_\mathcal{H}d$. By Definition \ref{inverse along element}, we conclude that $a^\core$ is inverse of $a$ along $aa^*$.

  Conversely, suppose that there exists inverse of $a\in R^\dag$ along $aa^*$, denote it by $x$, and let us show that $a\in R^\core$ and $x=a^\core$. By Definition \ref{inverse along element} we have that
  \begin{equation}\label{temp1}
    xa^2a^*=aa^*=aa^*ax
  \end{equation}
  and there exists $w\in R$ such that
  \begin{equation}\label{temp2}
    x=aa^*w.
  \end{equation}
  It is sufficient to show that $x$ satisfies the equations given in Theorem \ref{idempotenti za core inverz} (iii). By (\ref{z}),  we have
  \begin{align*}
    ax&= aa^\dag ax=zaa^*ax=zaa^*=aa^\dag,
  \end{align*}
  so $(ax)^*=ax$. Now, $axa=aa^\dag a=a$. Also,
  \begin{align}
    &xa^2=xaaa^\dag a=xa^2(a^\dag a)^*=xa^2a^*(a^\dag)^*\stackrel{(\ref{temp1})}{=}aa^*(a^\dag)^*=a(a^\dag a)^*=a \label{temp3}\\
    &ax^2\stackrel{(\ref{temp2})}{=}axaa^*w=aa^*w=x \nonumber\\
    &xax=xaaa^*w\stackrel{(\ref{temp3})}{=}aa^*w\stackrel{(\ref{temp2})}{=}x.\nonumber
  \end{align}
  The proof is complete.

  (ii): This statement may be proved in the same manner as (i).
\end{proof}

\begin{theorem}
  Let $a\in R$. Then:
  \begin{enumerate}[{\rm (i)}]
    \item $a$ is core invertible if and only if there exists $(a,a^*)$-inverse of $a$. In this case $(a,a^*)$-inverse of $a$ coincides with core inverse of $a$.
    \item $a$ is dual core invertible if and only if there exists $(a^*,a)$-inverse of $a$. In this case $(a^*,a)$-inverse of $a$ coincides with dual core inverse of $a$.
  \end{enumerate}
\end{theorem}
\begin{proof}
  We will only show the statement (i) because the statement (ii) may be proved similarly.
  Suppose that $a\in R^\core$ and let $x=a^\core$, $b=a$ and $c=a^*$. By the properties of core inverse we obtain
  \begin{align*}
    &xab=xa^2=a=b \\
    &cax=a^*ax=a^*(ax)^*=(axa)^*=a^*=c \\
    &x=xax=ax^2ax=bx^2ax\in bRx \\
    &x=xax=x(ax)^*=xx^*a^*=xx^*c\in xRc.
  \end{align*}
  Therefore, by Definition \ref{bc inverse}, $x=a^\core$ is $(a,a^*)$-inverse of $a$.

  Conversely, suppose that $(a,a^*)$-inverse of $a$ exists, denote it by $x$, and let us show that $x$ satisfies equations given in Theorem \ref{idempotenti za core inverz} (iii). By Definition \ref{bc inverse},
  \begin{equation}\label{temp4}
    xa^2=a, \quad a^*ax=a^*
  \end{equation}
  and there exists $w\in R$ such that
  \begin{equation}\label{temp6}
    x=awx.
  \end{equation}
  We obtain
  \begin{equation*}
    ax\stackrel{(\ref{temp4})}{=}(a^*ax)^*x=(ax)^*ax,
  \end{equation*}
  so $(ax)^*=ax$.
  Also,
  \begin{align}
    &axa=(ax)^*a=(a^*ax)^*\stackrel{(\ref{temp4})}{=}(a^*)^*=a \label{temp5} \\
    &ax^2\stackrel{(\ref{temp6})}{=}axawx\stackrel{(\ref{temp5})}{=}awx=x \nonumber \\
    &xax\stackrel{(\ref{temp6})}{=}xaawx\stackrel{(\ref{temp4})}{=}awx=x. \nonumber
  \end{align}
  The proof of the theorem is complete.
\end{proof}

\section{The case of $R=\B(H)$}

Let $H$ be arbitrary Hilbert space. In this section we consider the group, MP, core and dual core inverse in the case when $R$ is $\B(H)$, the algebra of all bounded linear operators on $H$. Of course, $\B(H)$ is a $*$-ring and all results from previous sections stay valid in the present setting. When $H$ is $n$-dimensional Hilbert space then we can identify $\B(H)$ with $M_n$.
We will denote by $\R(A)$ and $\N(A)$ the range and null-space of $A\in\B(H)$, respectively. Recall that $A\in\B(H)$ is regular if and only if $\R(A)$ is closed, \cite{dra}.
The ascent and descent of linear operator $A:H\rightarrow H$ are defined by
$$asc(A)=\inf_{p\in N}\{\N(A^p)=\N(A^{p+1})\},\;\;dsc(A)=\inf_{p\in N}\{\R(A^p)=\R(A^{p+1})\}.$$
If they are finite, they are equal and their common value is $\text{ind}(A),$ the index of $A.$ Also, $H=\R(A^{\text{ind}(A)})\oplus\N(A^{\text{ind}(A)})$ and $\R(A^{\text{ind}(A)})$ is closed, see \cite{dra}.

Let $A,B\in \B(H)$ such that $A\B(H)\subseteq B\B(H)$. Then there exist $Z\in\B(H)$ such that $A=BZ$ so $\R(A)\subseteq\R(B)$. If $B$ is regular and $\R(A)\subseteq\R(B)$ then $A=BB^{(1)}A$, where $B^{(1)}$ is arbitrary inner inverse of $B$.
Indeed, for every $x\in H$ there exists $z\in H$ such that $Ax=Bz$ so $$BB^{(1)}Ax=BB^{(1)}Bz=Bz=Ax.$$
Similarly, if $\B(H)A\subseteq\B(H)B$ then $\N(B)\subseteq\N(A)$. Suppose that $B$ is regular and $\N(B)\subseteq\N(A)$. Since $B(I-B^{(1)}B)=0$ we have $A(I-B^{(1)}B)=0$ so $A=AB^{(1)}B$; hence $\B(H)A\subseteq\B(H)B$. Let us consider the idempotents $P$, $Q$ and $R$ introduced in theorems \ref{idempotenti za grup inverz}--\ref{idempotenti za dual core inverz}.

\begin{lemma}\label{lemma za P, Q i R}
  Let $A\in\B(H)$. Then:
  \begin{enumerate}[{\rm (i)}]
    \item There exists self-adjoint idempotent $P$ such that $P\B(H)=A\B(H)$ if and only if $\R(A)$ is closed.
    \item There exists self-adjoint idempotent $R$ such that $\B(H)R=\B(H)A$ if and only if $\R(A)$ is closed.
    \item There exists idempotent $Q$ such that $Q\B(H)=A\B(H)$ and $\B(H)Q=\B(H)A$ if and only if $\ind(A)\leq 1$.
  \end{enumerate}
\end{lemma}
\begin{proof}
   In the proof we use the observation stated before theorem. Note that every idempotent is regular.

  (i): If there exists self-adjoint idempotent $P$ such that $P\B(H)=A\B(H)$ then $\R(A)=\R(P)$, so $\R(A)$ is closed. Conversely, if $\R(A)$ is closed then $A$ is regular and there exists self-adjoint idempotent $P$ such that $\R(P)=\R(A)$ and $\N(P)=\R(A)^\bot=\N(A^*)$. It follows that $P\B(H)=A\B(H)$.

  (ii): This part follows by (i) because $\R(A)$ is closed if and only if $\R(A^*)$ is closed and $\B(H)R=\B(H)A$ if and only if $R\B(H)=A^*\B(H).$

  (iii): If there exists idempotent $Q$ such that $Q\B(H)=A\B(H)$ and $\B(H)Q=\B(H)A$ then $A=QA=AQ$ and there exist $U,V\in\B(H)$ such that $Q=AU=VA$. Therefore, $A=AQ=AAU$ and $A=QA=VAA$ so $\R(A)=\R(A^2)$ and $\N(A)=\N(A^2)$. Hence, $\ind(A)\leq 1$. If $\ind(A)\leq 1$ then $\R(A)$ is closed, hence regular, and $$H=\R(A)\oplus\N(A).$$ Therefore, there exists an idempotent $Q$ such that $\R(Q)=\R(A)$ and $\N(Q)=\N(A)$. It follows that $Q\B(H)=A\B(H)$ and $\B(H)Q=\B(H)A$.
\end{proof}

According to theorems \ref{idempotenti za grup inverz} and \ref{idempotenti za MP inverz} and Lemma \ref{lemma za P, Q i R}, we can conclude the well known facts that $A$ has MP inverse if and only if $\R(A)$ is closed and $A$ has group inverse if and only if $\ind(A)\leq 1$.
But, it also follows that $A$ has core and/or dual core inverse if and only if $\ind(A)\leq 1$. Hence, we generalized the result which is known for complex matrices, see \cite{BT}.

Let $s,t\in R$ are idempotents and let $a\in R$ has representation $a=\bmatrix a_{11} & a_{12} \\ a_{21} & a_{22}\endbmatrix_{s\times t}$, where $a_{11}=sat$ and so on. Suppose now that $S,T\in \B(H)$ are idempotents. Since $\B(H)$ is a ring, an arbitrary bounded operator $A\in\B(H)$ has analogous representation $A=[A_{ij}]_{S\times T}$. Set $S_1=S$, $S_2=I-S$, $T_1=T$, $T_2=I-T$ and let us define the operators
$$A_{ij}':\R(T_j)\rightarrow \R(S_i), \quad A_{ij}'x=S_iAT_jx=A_{ij}x, \, x\in\R(T_j).$$
Therefore, for $x=x_1+x_2\in \R(T_1)\oplus \R(T_2)$ we have
$$Ax=A_{11}'x_1+A_{12}'x_2+A_{21}'x_1+A_{22}'x_2,$$
or in the matrix form
$$A=\bmatrix A_{11}' & A_{12}' \\ A_{21}' & A_{22}'\endbmatrix : \bmatrix \R(T) \\ \N(T) \endbmatrix \rightarrow \bmatrix \R(S) \\ \N(S) \endbmatrix.$$
It is not difficult to show that $A$ is bounded if and only if $A_{ij}'$, $i,j=1,2$ is bounded, see \cite{Rakic}.

Suppose that $A\in B(H)$ has a close range. From the proof of Theorem \ref{lemma za P, Q i R} it follows that self-adjoint idempotent $P$ induces decomposition $H=\R(A)\oplus\N(A^*)$. Self-adjoint idempotent $R$ induces decomposition $H=\R(A^*)\oplus \N(A)$. If $\ind(A)\leq 1$ then there exists idempotent $Q$ and it induces decomposition $H=\R(A)\oplus\N(A)$. Similarly, $Q^*$ induces $H=\R(A^*)\oplus\N(A^*)$. We are able to restate the representations given in (\ref{matricne forme prosirene}). If $A\in\B(H)$ and $\ind(A)\leq 1$ then $A$ has the following representations:
\begin{align*}
  &A=\bmatrix A_1 & 0 \\ 0 & 0\endbmatrix : \bmatrix \R(A) \\ \N(A) \endbmatrix \rightarrow \bmatrix \R(A) \\ \N(A) \endbmatrix  \quad \quad A=\bmatrix A_1 & 0 \\ 0 & 0\endbmatrix : \bmatrix \R(A) \\ \N(A) \endbmatrix \rightarrow \bmatrix \R(A) \\ \N(A^*) \endbmatrix \\
  &A=\bmatrix A_2 & 0 \\ 0 & 0\endbmatrix : \bmatrix \R(A^*) \\ \N(A) \endbmatrix \rightarrow \bmatrix \R(A) \\ \N(A) \endbmatrix  \quad \quad A=\bmatrix A_2 & 0 \\ 0 & 0\endbmatrix : \bmatrix \R(A^*) \\ \N(A) \endbmatrix \rightarrow \bmatrix \R(A) \\ \N(A^*) \endbmatrix ,
\end{align*}
where $A_1x=QAQx=PAQx$, $x\in \R(A)$ and $A_2x=QARx=PARx$, $x\in \R(A^*)$. The remark after (\ref{matricne forme prosirene}), in present setting, actually means that $A_1$ and $A_2$ are invertible. The other representations may be restated similarly. For example, for $A^\core$ we have
\begin{align*}
  &A^\core=\bmatrix A_1^{-1} & 0 \\ 0 & 0\endbmatrix : \bmatrix \R(A) \\ \N(A^*) \endbmatrix \rightarrow \bmatrix \R(A) \\ \N(A) \endbmatrix  \quad \quad A^\core=\bmatrix A_1^{-1} & 0 \\ 0 & 0\endbmatrix : \bmatrix \R(A) \\ \N(A^*) \endbmatrix \rightarrow \bmatrix \R(A) \\ \N(A^*) \endbmatrix \\
  &A^\core=\bmatrix A_3 & 0 \\ 0 & 0\endbmatrix : \bmatrix \R(A^*) \\ \N(A^*) \endbmatrix \rightarrow \bmatrix \R(A) \\ \N(A) \endbmatrix  \quad \quad A^\core=\bmatrix A_3 & 0 \\ 0 & 0\endbmatrix : \bmatrix \R(A^*) \\ \N(A^*) \endbmatrix \rightarrow \bmatrix \R(A) \\ \N(A^*) \endbmatrix ,
\end{align*}
where $A_3:\R(A^*)\rightarrow \R(A)$ is invertible bounded operator, $A_3x=QA^\core Q^*x=PA^\core Q^*x$, $x\in\R(A^*)$.

\vspace{1cm}

\textbf{Acknowledgment}

\medskip


\vspace{0.5cm}
{\noindent}Dragan S. Raki\'c\\
Faculty of Mechanical Engineering,\\
University of Ni\v s,\\
Aleksandra Medvedeva 14, 18000 Ni\v s, Serbia\\
{\noindent} {\it E-mail:} {\tt rakic.dragan@gmail.com}

\vspace{0.5cm}
{\noindent}Neboj\v{s}a \v{C}. Din\v{c}i\'c and Dragan S. Djordjevi\' c\\
Faculty of Sciences and Mathematics,\\
University of Ni\v s, \\
P.O. Box 95,
18000 Ni\v s, Serbia

{\noindent}{\it E-mail}: {\tt ndincic@hotmail.com}

{\noindent}{\it E-mail}: {\tt dragandjordjevic70@gmail.com}

\end{document}